\documentclass[12pt, reqno]{amsart}
\usepackage{amsmath, amsthm, amscd, amsfonts, amssymb, graphicx, color}
\usepackage[bookmarksnumbered, colorlinks, plainpages]{hyperref}
\hypersetup{colorlinks=true,linkcolor=red, anchorcolor=green, citecolor=cyan, urlcolor=red, filecolor=magenta, pdftoolbar=true}

\newtheorem{theorem}{Theorem}[section]
\newtheorem{lemma}[theorem]{Lemma}
\newtheorem{proposition}[theorem]{Proposition}
\newtheorem{corollary}[theorem]{Corollary}
\theoremstyle{definition}

\theoremstyle{remark}
\newtheorem{remark}[theorem]{Remark}
\numberwithin{equation}{section}

\begin{document}

\setcounter{page}{1}

\title[Some extensions of Berezin number inequalities ]{Some extensions of Berezin number inequalities on operators}

\author[M. Bakherad$^{*}$ , M. Hajmohamadi, R. Lashkaripour, \MakeLowercase{and} S. Sahoo]{Mojtaba Bakherad$^1$, Monire Hajmohamadi$^2$, Rahmatollah Lashkaripour$^3$   \MakeLowercase{and} Satyajit Sahoo$^4$ }

\address{$^{1,2,3}$Department of Mathematics, University of Sistan and Baluchestan, Zahedan, I.R.Iran. $^{4}$ P.G. Department of Mathematics, Utkal University, Vanivihar, Bhubaneswar-751004, India}
\email{ mojtaba.bakherad@yahoo.com and monire.hajmohamadi@yahoo.com,}
\email{  lashkari@hamoon.usb.ac.ir and satyajitsahoo2010@gmail.com}


\subjclass[2010]{Primary 47A63; Secondary 15A60}

\keywords{Berezin number; Berezin symbol; diagonal operator matrices, off-diagonal part.}


\begin{abstract}
In this paper, we establish some upper bounds for Berezin number inequalities including of $2\times 2$ operator
matrices and their off-diagonal parts. Among other inequalities, it is shown that if $T=\left[\begin{array}{cc}
 0&X\\
 Y&0
 \end{array}\right]$, then
 \begin{align*}
 \textbf{ber}^{r}(T)&\leq 2^{r-2}\left(\textbf{ber}(f^{2r}(|X|)+g^{2r}(|Y^*|))+\textbf{ber}(f^{2r}(|Y|)+g^{2r}(|X^*|))\right)\\&
 \,\,\,\,\,\,\,\,-2^{r-2}\inf_{\|(k_{\lambda_{1}},k_{\lambda_{2}})\|=1} \eta(k_{\lambda_{1}},k_{\lambda_{2}}),
 \end{align*}
 where
 {\footnotesize\begin{align*}
& \eta (k_{\lambda_{1}},k_{\lambda_{2}})=\\&\left(\left\langle(f^{2r}(|X|)+g^{2r}(|Y^*|)\right)k_{\lambda_{2}},k_{\lambda_{2}}\right\rangle^\frac{1}{2}-\left\langle \left(f^{2r}(|Y|)+g^{2r}(|X^*|)\right)k_{\lambda_{1}},k_{\lambda_{1}}\right\rangle^\frac{1}{2})^2,
\end{align*}}
 $X, Y$ are bounded linear operators on a Hilbert space $\mathcal H=\mathcal H(\Omega)$, $r\geq 1$ and $f$, $g$ are nonnegative  continuous  functions on $[0, \infty)$
 satisfying the relation $f(t)g(t)=t\,(t\in[0, \infty))$.

2010 Mathematics Subject Classification. Primary 47A30; Secondary 47A12, 15A60, 30E20.
\end{abstract} \maketitle

\section{Introduction and preliminaries}
A functional Hilbert space $\mathcal H=\mathcal H(\Omega)$ is a Hilbert space of complex valued functions on a (nonempty) set $\Omega$, which has the property that point evaluations are continuous i.e. for each $\lambda\in \Omega$ the map $f\mapsto f(\lambda)$ is a continuous linear functional on $\mathcal H$. The Riesz representation theorem ensure that for each $\lambda\in \Omega$ there is a unique element $k_{\lambda}\in \mathcal H$ such that $f(\lambda)=\langle f,k_{\lambda}\rangle$, for all $f\in \mathcal H$. The collection $\{k_{\lambda} : \lambda\in \Omega\}$ is called the reproducing kernel of $\mathcal H$. If $\{e_{n}\}$ is an orthonormal basis for a functional Hilbert space $\mathcal H$, then the reproducing kernel of $\mathcal H$ is given by $k_{\lambda}(z)=\sum_n\overline{e_{n}(\lambda)}e_{n}(z)$; . For $\lambda\in \Omega$, let $\hat{k_{\lambda}}=\frac{k_{\lambda}}{\|k_{\lambda}\|}$ be the normalized reproducing kernel of $\mathcal H$. For a bounded linear operator $A$ on $\mathcal H$, the function $\widetilde{A}$ defined on $\Omega$ by $\widetilde{A}(\lambda)=\langle A\hat{k_{\lambda}},\hat{k_{\lambda}}\rangle$ is the Berezin symbol of $A$, which firstly have been introduced by Berezin . The Berezin set and the Berezin number of the operator A are defined by
 \begin{align*}
 \textbf{Ber}(A):=\{\widetilde{A}(\lambda): \lambda\in \Omega\} \qquad \textrm{and} \qquad \textbf{ber}(A):=\sup\{|\widetilde{A}(\lambda)|: \lambda\in\Omega\},
 \end{align*}
respectively. In some recent papers,  several Berezin number inequalities have been investigated by authors
.
The Berezin number of  operators $A$ and $B$ satisfies the properties;
$\textbf{ ber}(\alpha A)=|\alpha|\textbf{ber}(A)\,\,(\alpha\in \mathbb C)$, and
$\textbf{ber}(A+B)\leq \textbf{ber}(A)+\textbf{ber}(B)$ and $\textbf{ber}(A)\leq\|A\|$, where $\|.\|$ is the operator norm.\\
Throughout this paper, the operator matrix $T=\left[\begin{array}{cc}
 S&X\\
 Y&R
 \end{array}\right]$ is a matrix, which $S\in{\mathcal B}({\mathcal H_1}), X\in{\mathcal B}({\mathcal H_2,\mathcal H_1})$, $Y\in{\mathcal B}({\mathcal H_1,\mathcal H_2})$ and $R\in{\mathcal B}({\mathcal H_2})$.\\
The authors in  showed an upper bound for off-diagonal part of operator matrix as follows:
\begin{align}
(\textbf{ber}(T))^{r}\leq \frac{1}{4}\|h(f^{2}(|Y|))+h(g^{2}(|Y|))\|+\frac{1}{4}\|h(f^{2}(|X|))+h(g^{2}(|X|))\|,
\end{align}
where $T=\left[\begin{array}{cc}
 0&X\\
 Y&0
 \end{array}\right]$ and  $f$, $g$ are nonnegative  continuous  functions on $[0, \infty)$ such that $f(t)g(t)=t\,\,(t\in [0, \infty))$.\\

 The classical Young inequality says that if $p, q>1$ such that $\frac{1}{p}+\frac{1}{q}=1$, then $ab\leq \frac{a^{p}}{p}+\frac{b^{q}}{q}$ for positive real numbers $a, b$. In, the authors showed that a refinement of the scalar Young inequality  as follows:
 \begin{align*}
 \left(a^{\frac{1}{p}}b^{\frac{1}{q}}\right)^{m}+r_{0}^{m}\left(a^{\frac{m}{2}}-b^{\frac{m}{2}}\right)^{2}\leq\left(\frac{a}{p}+\frac{b}{q}\right)^{m},
 \end{align*}
where $r_{0}=\min \{ \frac{1}{p}, \frac{1}{q}\}$ and $m=1, 2,\ldots$. In particular, if $p=q=2$, then
\begin{align}
\left(a^{\frac{1}{2}}b^{\frac{1}{2}}\right)^{m}+\left(\frac{1}{2}\right)^{m}\left(a^{\frac{m}{2}}-b^{\frac{m}{2}}\right)^{2}\leq 2^{-m}(a+b)^{m}.
\end{align}

 In this paper, we obtain some upper bounds for Berezin number inequalities of off-diagonal part of operator matrix and refinements of them. Moreover, we obtain Berezin number inequalities for diagonal operator matrix.

\section{Main results}

\bigskip To prove our results, we need the following lemmas.
\begin{lemma}
Let  $S\in {\mathcal B}({\mathcal H_1}(\Omega))$, $X\in {\mathcal B}({\mathcal H_2(\Omega), \mathcal H_1(\Omega)})$, $Y\in {\mathcal B}({\mathcal H_1(\Omega),\mathcal H_2(\Omega)})$ and $R\in {\mathcal B}({\mathcal H_2}(\Omega))$. Then the following statements hold:\\
$(a)\,\,\textbf{ber}\left(\left[\begin{array}{cc}
 S&0\\
 0&R
\end{array}\right]\right)$
$\leq \max \{\textbf{ber}(S), \textbf{ber}(R)\};$
\\
\\
$(b)\,\,\textbf{ber}\left(\left[\begin{array}{cc}
0&X\\
Y&0
\end{array}\right]\right)$
$\leq$ $ \frac{1}{2}(\|X\|+\|Y\|).$\\
In particular,
\begin{align}
\textbf{ber}\left(\left[\begin{array}{cc}
0&X\\
X&0
\end{array}\right]\right)
\leq \|X\|.
\end{align}
$(c)\,\,\textbf{ber}(S)=\sup_{\theta\in \mathbb{R}}\textbf{ber}(\texttt{Re}(e^{i\theta}S))$.
\end{lemma}

\bigskip The next lemma follows from the spectral theorem for positive operators and the Jensen's inequality.
\begin{lemma}
 Let $T\in{\mathcal B}({\mathcal H})$, $ T \geq 0$ and $x\in {\mathcal H}$ such that $\|x\|\leq1$. Then\\
$(a)\,\, \left\langle Tx, x\right\rangle^{r} \leq  \left\langle T^{r}x, x\right\rangle$ for $ r\geq 1;$\\
$(b)\,\,\left\langle T ^{r}x, x\right\rangle  \leq  \left\langle Tx, x\right\rangle^{r}$ for $ 0<r\leq 1$.\\
\end{lemma}
\begin{lemma}
Let $T\in{\mathcal B}({\mathcal H})$ and $x, y\in {\mathcal H}$ be any vectors.  If $f$, $g$ are nonnegative  continuous functions on $[0, \infty)$ which are satisfying the relation $f(t)g(t)=t\,(t\in[0, \infty))$, then
\begin{align*}
| \left\langle Tx, y \right\rangle |^2 \leq \left\langle f^2(|T |)x ,x \right\rangle\, \left\langle g^2(| T^{*}|)y,y\right\rangle.
 \end{align*}
 \end{lemma}
 Now, we are in position to demonstrate the main results of this section by
using some ideas from 
\begin{theorem}
Let
$T=\left[\begin{array}{cc}
 0&X\\
 Y&0
 \end{array}\right]$,  $r\geq 1$ and $f$, $g$ be nonnegative  continuous  functions on $[0, \infty)$ satisfying the relation $f(t)g(t)=t\,(t\in[0, \infty))$. Then
 \begin{align*}
 (\textbf{ber}(T))^{r}\leq \frac{2^r}{2}\textbf{ber}^{\frac{1}{2}}(f^{2r}(|X|)+g^{2r}(|Y^*|))\textbf{ber}^{\frac{1}{2}}(f^{2r}(|Y|)+g^{2r}(|X^*|)),
   \end{align*}
   and \begin{align*}
 (\textbf{ber}(T))^{r}\leq \frac{2^r}{2}\textbf{ber}^{\frac{1}{2}}(f^{2r}(|X|)+f^{2r}(|Y^*|))\textbf{ber}^{\frac{1}{2}}(g^{2r}(|Y|)+g^{2r}(|X^*|)).
   \end{align*}
 \end{theorem}
\begin{proof}
For any $(\lambda_{1},\lambda_{2})\in \Omega_{1}\times\Omega_{2}$, let $\hat{k}_{(\lambda_{1},\lambda_{2})}=\left[\begin{array}{cc}
 k_{\lambda_{1}}\\
 k_{\lambda_{2}}
 \end{array}\right]$ be the normalized reproducing kernel in $\mathcal H(\Omega_{1})\oplus\mathcal H(\Omega_{2})$. Then
 \begin{align*}
 &|\left\langle T\hat{k}_{(\lambda_{1},\lambda_{2})},\hat{k}_{(\lambda_{1},\lambda_{2})} \right\rangle |^{r}\\&
 =|\left\langle Xk_{\lambda_{2}}, k_{\lambda_{1}} \right\rangle+\left\langle Y k_{\lambda_{1}}, k_{\lambda_{2}} \right\rangle |^{r}\\&
 \leq\left(|\left\langle Xk_{\lambda_{2}}, k_{\lambda_{1}} \right\rangle|+|\left\langle Y k_{\lambda_{1}}, k_{\lambda_{2}} \right\rangle |\right)^{r} \qquad\qquad (\textrm {by the triangular inequality})\\&
 \leq\frac{2^r}{2}\left(|\left\langle Xk_{\lambda_{2}}, k_{\lambda_{1}} \right\rangle|^r+|\left\langle Yk_{\lambda_{1}}, k_{\lambda_{2}} \right\rangle |^{r}\right)
 \qquad (\textrm {by the convexity of\,} f(t)=t^r)\\&
 \leq\frac{2^r}{2}\Big(\left(\left\langle f^2(|X|)k_{\lambda_{2}}, k_{\lambda_{2}} \right\rangle^\frac{1}{2}\left\langle g^2(|X^*|)k_{\lambda_{1}}, k_{\lambda_{1}} \right\rangle^\frac{1}{2}\right)^r
 \\&\qquad+\left(\left\langle f^2(|Y|)k_{\lambda_{1}}, k_{\lambda_{1}} \right\rangle^\frac{1}{2}\left\langle g^2(|Y^*|)k_{\lambda_{2}}, k_{\lambda_{2}} \right\rangle^\frac{1}{2} \right)^{r}\Big)
\qquad(\textrm {by Lemma\,\,})\\&
\leq\frac{2^r}{2}\big(\left\langle f^{2r}(|X|)k_{\lambda_{2}}, k_{\lambda_{2}} \right\rangle^\frac{1}{2}\left\langle g^{2r}|X^*|k_{\lambda_{1}}, k_{\lambda_{1}} \right\rangle^\frac{1}{2}\\&
 +\left\langle f^{2r}(|Y|)k_{\lambda_{1}}, k_{\lambda_{1}} \right\rangle^\frac{1}{2}\left\langle g^{2r}(|Y^*|)k_{\lambda_{2}}, k_{\lambda_{2}} \right\rangle^\frac{1}{2}\big)\,\,\qquad\qquad(\textrm {by Lemma\,\,})
  \end{align*}
 \begin{align*}
 &\leq\frac{2^r}{2}\left(\left\langle f^{2r}(|X|)k_{\lambda_{2}}, k_{\lambda_{2}} \right\rangle+\left\langle g^{2r}(|Y^*|)k_{\lambda_{2}}, k_{\lambda_{2}} \right\rangle\right)^\frac{1}{2}\\&\,\,\,\times
 \left(\left\langle f^{2r}(|Y|)k_{\lambda_{1}}, k_{\lambda_{1}} \right\rangle+\left\langle g^{2r}(|X^*|)k_{\lambda_{1}}, k_{\lambda_{1}} \right\rangle\right)^\frac{1}{2}\\&
  \qquad\qquad\qquad\qquad (\textrm {by the Cauchy-Schwarz inequality})\\&
  =\frac{2^r}{2}\left\langle (f^{2r}(|X|)+g^{2r}(|Y^*|))k_{\lambda_{2}}, k_{\lambda_{2}} \right\rangle^\frac{1}{2} \left\langle (f^{2r}(|Y|)+g^{2r}(|X^*|))k_{\lambda_{1}}, k_{\lambda_{1}} \right\rangle^\frac{1}{2} \\&
  \leq \frac{2^r}{2}\textbf{ber}^{\frac{1}{2}}(f^{2r}(|X|)+g^{2r}(|Y^*|))\textbf{ber}^{\frac{1}{2}}(f^{2r}(|Y|)+g^{2r}(|X^*|)).
 \end{align*}
 Therefore
 \begin{align*}
 \textbf{ber}^{r}(T)\leq\frac{2^r}{2}\textbf{ber}^{\frac{1}{2}}(f^{2r}(|X|)+g^{2r}(|Y^*|))\textbf{ber}^{\frac{1}{2}}(f^{2r}(|Y|)+g^{2r}(|X^*|)).
 \end{align*}
 Hence, we get the first inequality. For the proof of the second inequality, we have
 \begin{align}
 &\left|\left\langle T\hat{k}_{(\lambda_{1},\lambda_{2})}, \hat{k}_{(\lambda_{1},\lambda_{2})} \right\rangle \right|^{r}\nonumber\\&
 =|\left\langle Xk_{\lambda_{2}}, k_{\lambda_{1}} \right\rangle+\left\langle Yk_{\lambda_{1}}, k_{\lambda_{2}} \right\rangle |^{r}\nonumber\\&
 \leq\left(|\left\langle Xk_{\lambda_{2}}, k_{\lambda_{1}} \right\rangle|+|\left\langle Yk_{\lambda_{1}}, k_{\lambda_{2}} \right\rangle |\right)^{r} \qquad\qquad (\textrm {by the triangular inequality})\nonumber\\&
 \leq\frac{2^r}{2}\left(|\left\langle Xk_{\lambda_{2}}, k_{\lambda_{1}} \right\rangle|^r+|\left\langle Yk_{\lambda_{1}}, k_{\lambda_{2}} \right\rangle |^{r}\right)
 \qquad (\textrm {by the convexity of\,} f(t)=t^r)\nonumber\\&
 \leq\frac{2^r}{2}\left(\left(\left\langle f^2(|X|)k_{\lambda_{2}}, k_{\lambda_{2}} \right\rangle^\frac{1}{2}\left\langle g^2(|X^*|)k_{\lambda_{1}}, k_{\lambda_{1}} \right\rangle^\frac{1}{2}\right)^r\right.\nonumber
\\&\qquad +\left.\left(\left\langle g^2(|Y|)k_{\lambda_{1}}, k_{\lambda_{1}} \right\rangle^\frac{1}{2}\left\langle f^2(|Y^*|)k_{\lambda_{2}}, k_{\lambda_{2}} \right\rangle^\frac{1}{2} \right)^{r}\right)
\qquad(\textrm {by Lemma\,\,}).
 \end{align}
  With  a similar argument to the proof of the first inequality we have the second inequality and this completes the proof of the theorem.
\end{proof}
\bigskip Theorem  includes some special cases as follows.
\begin{corollary}
Let
$T=\left[\begin{array}{cc}
 0&X\\
 Y&0
 \end{array}\right]\in {\mathcal B}({\mathcal H_1\oplus\mathcal H_2})$, $0\leq p\leq1$ and $r\geq1$. Then
 \begin{align*}
 \textbf{ber}^{r}(T)\leq 2^{r-2}  \textbf{ber}^{\frac{1}{2}} (| X |^{2rp} +   | Y^* |^{2r(1-p)})\textbf{ber}^{\frac{1}{2}}(|Y |^{2rp} +  | X^* |^{2r(1-p)})
   \end{align*}
   and
\begin{align*}
 \textbf{ber}^{r}(T)\leq 2^{r-2}\textbf{ber}^{\frac{1}{2}} (| X |^{2rp} +   | Y^* |^{2rp})\textbf{ber}^{\frac{1}{2}}(|Y |^{2r(1-p)} + | X^* |^{2r(1-p)}).
   \end{align*}
  \end{corollary}
  \begin{proof}
  The result follows immediately from Theorem  for $f(t)=t^p$ and $g(t)=t^{1-p}\,\,(0\leq p\leq1)$.
  \end{proof}
 \begin{remark} Taking $f(t)=g(t)=t^{\frac{1}{2}}\,(t\in[0,\infty))$ and $r=1$ in Theorem , we get the following result:
 \begin{align*}
 \textbf{ber}(\left[\begin{array}{cc}
 0&X\\
 Y&0
 \end{array}\right])\leq \frac{1}{2} \textbf{ber}^{\frac{1}{2}}(| X | + | Y^* |)\textbf{ber}^{\frac{1}{2}}(|Y| + |X^*|).
   \end{align*}
 \end{remark}

  \bigskip If we put $Y=X$, $r=1$ and $f(t)=g(t)=t^{\frac{1}{2}}$ in Theorem , then we get a refinement of inequality  as follows.
 \begin{corollary}
Let
$X\in {\mathcal B}({\mathcal H})$,  $r\geq 1$ and $f$, $g$ be nonnegative  continuous  functions on $[0, \infty)$ satisfying the relation $f(t)g(t)=t\,(t\in[0, \infty))$. Then
 \begin{align*}
 \textbf{ber}\left(\left[\begin{array}{cc}
 0&X\\
 X&0
 \end{array}\right]\right)\leq \frac{1}{2}\textbf{ber}(|X|+|X^*|)\leq\|X\|
   \end{align*}
 \end{corollary}
 In the following corollary we obtain refinement of .
 \begin{corollary}
 Let
$T=\left[\begin{array}{cc}
 0&X\\
 Y&0
 \end{array}\right]$. Then
 \begin{align*}
 \textbf{ber}(T)&\leq\frac{1}{2}\textbf{ber}^{\frac{1}{2}}(|X|+|Y^*|)\textbf{ber}^{\frac{1}{2}}(|Y|+|X^*|)\\&\leq \frac{1}{2}\max\{\textbf{ber}(|X|+|Y^*|),\textbf{ber}(|Y|+|X^*|)\}.
   \end{align*}
 \end{corollary}
 \begin{proof}
 In Theorem , we putting $r=1$, $f(t)=g(t)=t^{\frac{1}{2}}$ and applying arithmetic-geometric mean, and get desired result as follows:
 \begin{align*}
 \textbf{ber}(T)&\leq\frac{1}{2}\textbf{ber}^{\frac{1}{2}}(|X|+|Y^*|)\textbf{ber}^{\frac{1}{2}}(|Y|+|X^*|)\\&
 \leq\frac{1}{2}\left(\frac{\textbf{ber}(|X|+|Y^*|)+\textbf{ber}(|Y|+|X^*|)}{2}\right)\\&
 \leq \frac{1}{2}\max\{\textbf{ber}(|X|+|Y^*|),\textbf{ber}(|Y|+|X^*|)\}.
 \end{align*}
 This completes the proof.
 \end{proof}
  Applying inequality  we obtain the following theorem.
\begin{theorem}
Let
$T=\left[\begin{array}{cc}
 0&X\\
 Y&0
 \end{array}\right]\in {\mathcal B}({\mathcal H_1\oplus\mathcal H_2})$ and  $f$, $g$ be nonnegative  continuous  functions on $[0, \infty)$  satisfying the relation $f(t)g(t)=t$ $(t\in [0, \infty))$. Then for $r\geq 1$
\begin{align*}
 \textbf{ber}^{r}(T)&\leq 2^{r-2}\left(\textbf{ber}(f^{2r}(|X|)+g^{2r}(|Y^*|))+\textbf{ber}(f^{2r}(|Y|)+g^{2r}(|X^*|))\right)\\&
 \,\,\,\,\,\,\,\,-2^{r-2}\inf_{\|(k_{\lambda_{1}},k_{\lambda_{2}})\|=1} \eta(k_{\lambda_{1}},k_{\lambda_{2}}),
 \end{align*}
 where
 {\footnotesize\begin{align*}
 \eta (k_{\lambda_{1}},k_{\lambda_{2}})=\left(\left\langle \left(f^{2r}(|X|)+g^{2r}(|Y^*|)\right)k_{\lambda_{2}},k_{\lambda_{2}}\right\rangle^\frac{1}{2}-\left\langle \left(f^{2r}(|Y|)+g^{2r}(|X^*|)\right)k_{\lambda_{1}},k_{\lambda_{1}}\right\rangle^\frac{1}{2}\right)^2.
 \end{align*}}
 \end{theorem}
 \begin{proof}
 For any $(\lambda_{1},\lambda_{2})\in \Omega_{1}\times\Omega_{2}$, let $\hat{k}_{(\lambda_{1},\lambda_{2})}=\left[\begin{array}{cc}
 k_{\lambda_{1}}\\
 k_{\lambda_{2}}
 \end{array}\right]$ be the normalized reproducing kernel in $\mathcal H(\Omega_{1})\oplus\mathcal H(\Omega_{2})$. Then
   {\footnotesize\begin{align*}
 &|\langle T\hat{k}_{(\lambda_{1},\lambda_{2})}, \hat{k}_{(\lambda_{1},\lambda_{2})}\rangle|^{r}
 \\& = |\langle Xk_{\lambda_{2}},k_{\lambda_{1}}\rangle+\langle Yk_{\lambda_{1}}, k_{\lambda_{2}}\rangle|^{r}
 \\&\leq\left(|\langle Xk_{\lambda_{2}},k_{\lambda_{1}}\rangle|+|\langle Yk_{\lambda_{1}}, k_{\lambda_{2}}\rangle|\right)^{r}\qquad(\textrm {\normalsize{by the triangular inequality}})
 \\&\leq\frac{2^r}{2}\left(|\langle Xk_{\lambda_{2}},k_{\lambda_{1}}\rangle|^r+|\langle Yk_{\lambda_{1}}, k_{\lambda_{2}}\rangle|^r\right)\qquad(\textrm {\normalsize{by the convexity of}\,} f(t)=t^r)
    \\&\leq\frac{2^r}{2}\left(\langle f^2(|X|)k_{\lambda_{2}},k_{\lambda_{2}}\rangle^\frac{r}{2}\langle g^2(|X^*|)k_{\lambda_{1}},k_{\lambda_{1}}\rangle^\frac{r}{2}
  +\langle f^2(|Y|)k_{\lambda_{1}},k_{\lambda_{1}}\rangle^\frac{r}{2}\langle f^2(|Y^*|)k_{\lambda_{2}},k_{\lambda_{2}}\rangle^\frac{r}{2}\right)
  \\&\qquad\qquad\qquad\qquad\qquad\qquad\qquad\qquad(\textrm {\normalsize{by Lemma}\,\,})
  \\&\leq\frac{2^r}{2}\left(\langle f^{2r}(|X|)k_{\lambda_{2}},k_{\lambda_{2}}\rangle^\frac{1}{2}\langle g^{2r}(|X^*|)k_{\lambda_{1}},k_{\lambda_{1}}\rangle^\frac{1}{2}
  +\langle f^{2r}(|Y|)k_{\lambda_{1}},k_{\lambda_{1}}\rangle^\frac{1}{2}\langle g^{2r}(|Y^*|)k_{\lambda_{2}},k_{\lambda_{2}}\rangle^\frac{1}{2}\right)
  \\&\leq\frac{2^r}{2}\left(\langle f^{2r}(|X|)k_{\lambda_{2}},k_{\lambda_{2}}\rangle+\langle g^{2r}(|X^*|)k_{\lambda_{2}},k_{\lambda_{2}}\rangle\right)^\frac{1}{2}\left(f^{2r}(|Y|)k_{\lambda_{1}},k_{\lambda_{1}}\rangle+\langle g^{2r}(|Y^*|)k_{\lambda_{1}},k_{\lambda_{1}}\rangle\right)^\frac{1}{2}
  \\&=\frac{2^r}{2}\langle \left(f^{2r}(|X|)+g^{2r}(|Y^*|)\right)k_{\lambda_{2}},k_{\lambda_{2}}\rangle^\frac{1}{2}\langle \left(f^{2r}(|Y|)+g^{2r}(|X^*|)\right)k_{\lambda_{1}},k_{\lambda_{1}}\rangle^\frac{1}{2}\\&
  \leq\frac{2^r}{4}\left(\langle \left(f^{2r}(|X|)+g^{2r}(|Y^*|)\right)k_{\lambda_{2}},k_{\lambda_{2}}\rangle+\langle \left(f^{2r}(|Y|)+g^{2r}(|X^*|)\right)k_{\lambda_{1}},k_{\lambda_{1}}\rangle\right)\\&
  \,\,\,\,\,-\frac{2^r}{4}\left(\langle \left(f^{2r}(|X|)+g^{2r}(|Y^*|)\right)k_{\lambda_{2}},k_{\lambda_{2}}\rangle^\frac{1}{2}-\langle \left(f^{2r}(|Y|)+g^{2r}(|X^*|)\right)k_{\lambda_{1}},k_{\lambda_{1}}\rangle^\frac{1}{2}\right)^2\\&
  \qquad\qquad\qquad\qquad\qquad\qquad\qquad\qquad(\textrm {\normalsize{by inequality}\,\,})\\&
  \leq\frac{2^r}{4}\left(\textbf{ber}(f^{2r}(|X|)+g^{2r}(|Y^*|))+\textbf{ber}(f^{2r}(|Y|)+g^{2r}(|X^*|))\right)\\&
  \,\,\,\,\,-\frac{2^r}{4}\left(\langle \left(f^{2r}(|X|)+g^{2r}(|Y^*|)\right)k_{\lambda_{2}},k_{\lambda_{2}}\rangle^\frac{1}{2}-\langle \left(f^{2r}(|Y|)+g^{2r}(|X^*|)\right)k_{\lambda_{1}},k_{\lambda_{1}}\rangle^\frac{1}{2}\right)^2.
  \end{align*}}
 Taking the supremum over all unit vectors $\hat{k}_{(\lambda_{1},\lambda_{2})}=\left[\begin{array}{cc}
 k_{\lambda_{1}}\\
 k_{\lambda_{2}}
 \end{array}\right]
 \in {\mathcal H_1\oplus\mathcal H_2}$, we  get the desired inequality.
 \end{proof}

 \bigskip If we put $X=Y$ in Theorem , then we get next result.\\
 \begin{corollary}
Let $X\in {\mathbb B}({\mathcal H})$ and  $f$, $g$ be nonnegative  continuous  functions on $[0, \infty)$  satisfying the relation $f(t)g(t)=t$ $(t\in [0, \infty))$. Then for $r\geq 1$
{\footnotesize\begin{align*}
 \textbf{ber}^{r}(\left[\begin{array}{cc}
 0&X\\
 X&0
 \end{array}\right])\leq 2^{r-1}\|f^{2r}(|X|)+g^{2r}(|X^*|)\|-2^{r-2}\inf_{\|(k_{\lambda_{1}},k_{\lambda_{2}})\|=1} \eta (k_{\lambda_{1}},k_{\lambda_{2}}),
 \end{align*}}
 where
 {\footnotesize\begin{align*}
 \eta (k_{\lambda_{1}},k_{\lambda_{2}})=\left(\left\langle \left(f^{2r}(|X|)+g^{2r}(|X^*|)\right)k_{\lambda_{2}},k_{\lambda_{2}}\right\rangle^\frac{1}{2}-\left\langle \left(f^{2r}(|X|)+g^{2r}(|X^*|)\right)k_{\lambda_{1}},k_{\lambda_{1}}\right\rangle^\frac{1}{2}\right)^2.
 \end{align*}}
 \end{corollary}
 \section{generalizations of Berezin number an operator}
 In this section, we present some Berezin number inequalities for the generalized Aluthge transform, $\tilde{T}_t$, and then we present some inequalities, which generalized known inequalities.\\
 Let $T=U|T|$ (U is a partial isometry with $kerU=rang|T|^{\bot}$) be the polar decomposition of T.
 For an operator $T\in {\mathcal B}(\mathcal H)$, the generalized Aluthge transform, denoted by, $\tilde{T}_t$, is defined as
 \begin{align*}
 \tilde{T}_t=|T|^{t}U|A|^{1-t},\qquad(0\leq t\leq1).
 \end{align*}
 In the next theorem, we obtain an upper bound for the Berezin number of generalized Aluthge transform of the off-diagonal operator matrix $T=\left[\begin{array}{cc}
 0&X\\
 Y&0
 \end{array}\right]$.
 \begin{theorem}
Let
$T=\left[\begin{array}{cc}
 0&X\\
 Y&0
 \end{array}\right]\in {\mathbb B}({\mathcal H_1\oplus\mathcal H_2})$. Then
 \begin{align}
  \textbf{ber}(\tilde{T}_{t})\leq \frac{1}{2}(\||Y|^{t}|X^{*}|^{1-t}\|+\||X|^{t}|Y^{*}|^{1-t}\|).
 \end{align}
 \end{theorem}
 \begin{proof}
 Let $X=U|X|$ and $Y=V|Y|$ be the polar decompositions of the operators X and Y. Then
 \begin{align*}
 \left[\begin{array}{cc}
 0&X\\
 Y&0
 \end{array}\right]=\left[\begin{array}{cc}
 0&U\\
 V&0
 \end{array}\right]\left[\begin{array}{cc}
 |Y|&0\\
 0&|X|
 \end{array}\right]
 \end{align*}
 is the polar decomposition of T. The generalized Aluthge transform of T is
 \begin{align*}
 \tilde{T}_{t}=|T|^{t}\left[\begin{array}{cc}
 0&U\\
 V&0
 \end{array}\right]|T|^{1-t}&=\left[\begin{array}{cc}
 |Y|^{t}&0\\
 0&|X|^{t}
 \end{array}\right]\left[\begin{array}{cc}
 0&U\\
 V&0
 \end{array}\right]
 \left[\begin{array}{cc}
 |Y|^{1-t}&0\\
 0&|X|^{1-t}
 \end{array}\right]\\&
 =\left[\begin{array}{cc}
 0&|Y|^{t}U|X|^{1-t}\\
 |X|^{t}V|Y|^{1-t}&0
 \end{array}\right].
 \end{align*}
 So
 \begin{align*}
 \textbf{ber}(\tilde{T}_{t})&=\textbf{ber}\left(\left[\begin{array}{cc}
 0&|Y|^{t}U|X|^{1-t}\\
 |X|^{t}V|Y|^{1-t}&0
 \end{array}\right]\right)\\&
 \leq \frac{1}{2}(\||Y|^{t}U|X|^{1-t}\|+\||X|^{t}V|Y|^{1-t}\|).\qquad\qquad (\textrm {by Lemma })
 \end{align*}
 Since, $|X^{*}|^{2}=XX^{*}=U|X|^{2}U^{*}$, so $|X|^{1-t}=U^{*}|X^{*}|^{1-t}U$. Thus, $\||Y|^{t}U|X|^{1-t}\|=\||Y|^{t}UU^{*}|X^{*}|^{1-t}U\|=\||Y|^{t}|X^{*}|^{1-t}\|$. Similarly, $\||X|^{t}V|Y|^{1-t}\|=\||X|^{t}|Y^{*}|^{1-t}\|$. Therefore
 \begin{align*}
 \textbf{ber}(\tilde{T}_{t})\leq \frac{1}{2}(\||Y|^{t}|X^{*}|^{1-t}\|+\||X|^{t}|Y^{*}|^{1-t}\|).
 \end{align*}
 \end{proof}
\begin{theorem}
Let
$T\in {\mathcal B}({\mathcal H})$. Then
\begin{align}
\textbf{ber}(T)\leq \frac{1}{4}\| |T|^{2t}+|T|^{2(1-t)}\|+\frac{1}{2}\textbf{ber}(\tilde{T}_{t}),
\end{align}
where $t\in[0,1]$.
 \end{theorem}
 \begin{proof}
Let $\hat{k}_{\lambda}\in \mathcal{H}$. We have
 \begin{align*}
 &\texttt{Re}\langle e^{i\theta}T\hat{k}_{\lambda},\hat{k}_{\lambda}\rangle\\&=\texttt{Re}\langle e^{i\theta}U|T|\hat{k}_{\lambda},\hat{k}_{\lambda}\rangle\\&
 =\texttt{Re}\langle e^{i\theta}U |T|^{t}|T|^{1-t}\hat{k}_{\lambda},\hat{k}_{\lambda}\rangle\\&
 =\texttt{Re}\langle e^{i\theta}|T|^{1-t}\hat{k}_{\lambda},|T|^{t}U^{*}\hat{k}_{\lambda}\rangle\\&
 =\frac{1}{4}\|(e^{i\theta}|T|^{1-t}+|T|^{t}U^{*})\hat{k}_{\lambda}\|^{2}-\frac{1}{4}\|(e^{i\theta}|T|^{1-t}-|T|^{t}U^{*})\hat{k}_{\lambda}\|^{2}\\&
 \qquad\qquad\qquad (\textrm {by the polarization identity})\\&
 \leq \frac{1}{4}\|(e^{i\theta}|T|^{1-t}+|T|^{t}U^{*})\hat{k}_{\lambda}\|^{2}\\&
 =\frac{1}{4}\langle(e^{i\theta}|T|^{1-t}+|T|^{t}U^{*})\hat{k}_{\lambda},(e^{i\theta}|T|^{1-t}+|T|^{t}U^{*})\hat{k}_{\lambda}\rangle\\&
 =\frac{1}{4}\langle(e^{i\theta}|T|^{1-t}+|T|^{t}U^{*})(e^{-i\theta}|T|^{1-t}+U|T|^{t})\hat{k}_{\lambda},\hat{k}_{\lambda}\rangle\\&
 =\frac{1}{4}\langle|T|^{2t}+|T|^{2(1-t)}+e^{i\theta}\tilde{T}_{t}+e^{-i\theta}(\tilde{T}_{t})^{*}\hat{k}_{\lambda},\hat{k}_{\lambda}\rangle\\&
=\frac{1}{4}\langle|T|^{2t}+|T|^{2(1-t)}\hat{k}_{\lambda},\hat{k}_{\lambda}\rangle+
\frac{1}{4}\langle e^{i\theta}\tilde{T}_{t}+e^{-i\theta}(\tilde{T}_{t})^{*}\hat{k}_{\lambda},\hat{k}_{\lambda}\rangle\\&
=\frac{1}{4}\langle|T|^{2t}+|T|^{2(1-t)}\hat{k}_{\lambda},\hat{k}_{\lambda}\rangle+
\frac{1}{2}\langle\texttt{Re}(e^{i\theta}\tilde{T}_{t})\hat{k}_{\lambda},\hat{k}_{\lambda}\rangle\\&
\leq\frac{1}{4}\||T|^{2t}+|T|^{2(1-t)}\|+
\frac{1}{2}\textbf{ber}(\texttt{Re}(e^{i\theta}\tilde{T}_{t})\\&
 \leq\frac{1}{4}\||T|^{2t}+|T|^{2(1-t)}\|+\frac{1}{2}\textbf{ber}(\tilde{T}_{t}).
 \end{align*}
 By taking the supremum over $\lambda\in\Omega$, we get the desired result.
\end{proof}
\begin{remark}
By putting $t=\frac{1}{2}$ in Theorem , we get
$T\in {\mathcal B}({\mathcal H})$. Then
\begin{align}
\textbf{ber}(T)\leq \frac{1}{2}\| T\|+\frac{1}{2}\textbf{ber}(\tilde{T}_{t})
\end{align}
\end{remark}
From Theorem  we can get the next result for the off-diagonal operator matrix $T=\left[\begin{array}{cc}
 0&X\\
 Y&0
 \end{array}\right]$.
 \begin{corollary}
 Let
$T=\left[\begin{array}{cc}
 0&X\\
 Y&0
 \end{array}\right]\in {\mathcal B}({\mathcal H_1\oplus\mathcal H_2})$. Then
 \begin{align}
\textbf{ber}(T)\leq \frac{1}{2}\max(\|X\|,\|Y\|)+\frac{1}{4}(\||Y|^{t}|X^{*}|^{1-t}\|+\||X|^{t}|Y^{*}|^{1-t}\|).
 \end{align}
\end{corollary}
\begin{proof}
It is immediately deduced from Theorem .
\end{proof}
\begin{corollary}
Let
$T=\left[\begin{array}{cc}
 0&X\\
 Y&0
 \end{array}\right]\in {\mathcal B}({\mathcal H_1\oplus\mathcal H_2})$. Then
\begin{align}
\|X+Y\|\leq \max(\|X\|,\|Y\|)+\frac{1}{2}\left(\||X|^{1/2}|Y|^{1/2}\|+\||X^*|^{1/2}|Y^{*}|^{1/2}\|\right).
\end{align}
\end{corollary}
\begin{proof}
Applying Theorem, we have
\begin{align*}
\|X+Y^*\|&=\|T+T^{*}\|.
\end{align*}
\end{proof}

In the following we present some Berezin number inequalities for the operator matrix $T=\left[\begin{array}{cc}
	S&X\\
	Y&R
	\end{array}\right].$\\

\begin{theorem}
	Let
	$T=\left[\begin{array}{cc}
	S&X\\
	Y&R
	\end{array}\right]\in {\mathcal B}({\mathcal H_1\oplus\mathcal H_2})$.
	Then
	{\footnotesize\begin{align}
	\textbf{ber}(T)\leq \frac{1}{2}\textbf{ber}(S)+\textbf{ber}(R)+\frac{1}{2}\sqrt{\alpha^2\textbf{ber}^2(S)+\|X\|^2}+\frac{1}{2}\sqrt{(1-\alpha)^2\textbf{ber}^2(S)+\|Y\|^2}
	\end{align}}
	for $0\leq \alpha\leq 1$.
\end{theorem}

\begin{proof}
	\begin{align*}
		&\textbf{ber}(\texttt{Re}(e^{i\theta}T))\\&=	\textbf{ber}\left(\frac{e^{i\theta}T+e^{-i\theta}T^*}{2}\right)\\
		&=	\frac{1}{2}\textbf{ber}\left[\begin{array}{cc}
		2\texttt{Re}(e^{i\theta}S)&e^{i\theta}X+e^{-i\theta}Y^*\\
	e^{i\theta}Y+e^{-i\theta}X^*&2\texttt{Re}(e^{i\theta}R)
		\end{array}\right]\\
		&\leq\frac{1}{2} \bigg(\textbf{ber}\left[\begin{array}{cc}
		2\alpha\texttt{Re}(e^{i\theta}S)&e^{i\theta}X\\
		e^{-i\theta}X^*&0
		\end{array}\right]+\textbf{ber}\left[\begin{array}{cc}
		2(1-\alpha)\texttt{Re}(e^{i\theta}S)&e^{-i\theta}Y^*\\
		e^{i\theta}Y&0
		\end{array}\right]\\
		&+\textbf{ber}\left[\begin{array}{cc}
	0&0\\
	0&2\texttt{Re}(e^{i\theta}R)
		\end{array}\right]\bigg)\\
			&\leq\frac{1}{2} \bigg(\left[\begin{array}{cc}
		2\alpha\textbf{ber}(\texttt{Re}(e^{i\theta}S))&\|e^{i\theta}X\|\\
		\|e^{-i\theta}X^*\|&0
		\end{array}\right]+\left[\begin{array}{cc}
		2(1-\alpha)\textbf{ber}(\texttt{Re}(e^{i\theta}S))&\|e^{-i\theta}Y^*\|\\
		\|e^{i\theta}Y\|&0
		\end{array}\right]\\
		&+\left[\begin{array}{cc}
		0&0\\
		0&2\textbf{ber}(\texttt{Re}(e^{i\theta}R))
		\end{array}\right]\bigg)~~~~\,\,\,\,(\mbox{by Theorem 2.1})\\
		&\leq\frac{1}{2} \bigg(\left[\begin{array}{cc}
		2\alpha\textbf{ber}(S)&\|X\|\\
		\|X\|&0
		\end{array}\right]+\left[\begin{array}{cc}
		2(1-\alpha)\textbf{ber}(S)&\|Y\|\\
		\|Y\|&0
		\end{array}\right]
		+2\textbf{ber}(R)
		\bigg)\\
		&=\frac{1}{2}\bigg[\alpha\textbf{ber}(S)+\sqrt{\alpha^2\textbf{ber}^2(S)+\|X\|^2}+(1-\alpha)\textbf{ber}(S)\\
		&+\sqrt{(1-\alpha)^2\textbf{ber}^2(S)+\|Y\|^2}+2\textbf{ber}(R)\bigg]
		\\
		&\leq\frac{1}{2}\bigg[\textbf{ber}(S)+2\textbf{ber}(R)+\sqrt{\alpha^2\textbf{ber}^2(S)+\|X\|^2}
		+\sqrt{(1-\alpha)^2\textbf{ber}^2(S)+\|Y\|^2}\bigg].
	\end{align*}
	Taking supremum over $\theta\in \mathbb{R}$, we get
		\begin{align*}
	&\textbf{ber}(T)\\&\leq \frac{1}{2}\textbf{ber}(S)+\textbf{ber}(R)+\frac{1}{2}\sqrt{\alpha^2\textbf{ber}^2(S)+\|X\|^2}+\frac{1}{2}\sqrt{(1-\alpha)^2\textbf{ber}^2(S)+\|Y\|^2}.
	\end{align*}
\end{proof}

Using similar argument as used in previous theorem, we have the following result.

\begin{theorem}
	Let
	$T=\left[\begin{array}{cc}
	S&X\\
	Y&R
	\end{array}\right]\in {\mathcal B}({\mathcal H_1\oplus\mathcal H_2})$.
	Then
	\begin{align}
	&\textbf{ber}(T)\nonumber\\&\leq \frac{1}{2}\textbf{ber}(R)+\textbf{ber}(S)+\frac{1}{2}\sqrt{\alpha^2\textbf{ber}^2(R)+\|Y\|^2}+\frac{1}{2}\sqrt{(1-\alpha)^2\textbf{ber}^2(R)+\|X\|^2}
	\end{align}
	for $0\leq \alpha\leq 1$.
\end{theorem}

In the final part of the article we want to provide generalization of Berezin number of an operator. For our goal we need to the following inequalities, which obtain in :
\begin{align}
a^{\nu}b^{1-\nu}\leq \nu a+(1-\nu)b\leq\left(\nu a^r+(1-\nu)b^r\right)^{\frac{1}{r}},
\end{align}
and
\begin{align}
ab\leq\frac{a^{p}}{p}+\frac{b^{q}}{q}\leq(\frac{a^{pr}}{p}+\frac{b^{qr}}{q})^{\frac{1}{r}},
\end{align}
where $a,b\geq0, \nu\in[0,1], r\geq1$ and $p,q>1$ such that $\frac{1}{p}+\frac{1}{q}=1$.
\begin{proposition}
Let $T\in {\mathcal B}({\mathcal H})$. Then
\begin{align}
\textbf{ber}^{2r}(T)\leq\frac{1}{2}(\textbf{ber}^r(T^{2})+\|T\|^{2r})
\end{align}
for any $r\geq1$.
\end{proposition}
\begin{proof}
We need the following refinement of Schwarz's inequality:
\begin{align}
\|a\|\|b\|\geq|\langle a,b\rangle-\langle a,e\rangle\langle e,b\rangle|+|\langle a,e\rangle\langle e,b\rangle|\geq |\langle a,b\rangle|,
\end{align}
where $a,b,e$ are vectors in ${\mathcal H}$ and $\|e\|=1$.\\
From , we have
\begin{align}
\frac{1}{2}(\|a\|\|b\|+|\langle a,b\rangle|)\geq |\langle a,e\rangle\langle e,b\rangle|
\end{align}
By putting $e=k_{\lambda}$, $a=Tk_{\lambda}$ and $b=T^{*}k_{\lambda}$ in , we get
\begin{align*}
|\langle Tk_{\lambda},k_{\lambda}\rangle|^{2}&\leq \frac{1}{2}(\|Tk_{\lambda}\|\|T^*k_{\lambda}\|+|\langle T^2k_{\lambda},k_{\lambda}\rangle|)\\&
\leq \left(\frac{\|Tk_{\lambda}\|^r\|T^*k_{\lambda}\|^r+|\langle T^2k_{\lambda},k_{\lambda}\rangle|^r}{2}\right)^{\frac{1}{r}}.
\end{align*}
Hence
\begin{align*}
|\langle Tk_{\lambda},k_{\lambda}\rangle|^{2r}\leq\frac{1}{2}(\|Tk_{\lambda}\|^r\|T^*k_{\lambda}\|^r+|\langle T^2k_{\lambda},k_{\lambda}\rangle|^r).
\end{align*}
Taking the supremum over all unit vectors $k_{\lambda}$, we get the desired result.
\end{proof}
\begin{remark}
From inequality , we obviously have
\begin{align*}
\textbf{ber}^{2r}(T)\leq\frac{1}{2}(\textbf{ber}^r(T^{2})+\|T\|^{2r})&\leq\frac{1}{2}(\|T\|^{2r}+\|T\|^{2r})\\&
\leq \|T\|^{2r}.
\end{align*}
\end{remark}
In the next theorem, we obtain upper bound for powers $\textbf{ber}$.
\begin{theorem}
Let
$T\in {\mathcal B}({\mathcal H})$ and  $f$, $g$ be nonnegative  continuous  functions on $[0, \infty)$ satisfying the relation $f(t)g(t)=t\,\,(t\in [0, \infty))$. Then
 \begin{align}
\textbf{ber}^{2r}(T)\leq \frac{1}{2}(\|T\|^{2r}+\textbf{ber}(\frac{1}{p}f^{pr}(|T^{2}|)+\frac{1}{q}g^{qr}(|(T^2)^{*}|),
\end{align}
where $r\geq1, p\geq q>1$ with $\frac{1}{p}+\frac{1}{q}=1$ and $qr\geq2$.
\end{theorem}
\begin{proof}
Let $k_{\lambda}\in {\mathcal H}$ be unit vector. Then from Lemma , we have
\begin{align*}
|\langle T^{2}k_{\lambda},k_{\lambda}\rangle|^{r}&\leq \|f(|T^{2}|)k_{\lambda}\|^{r}\|g(|(T^{2})^{*}|)k_{\lambda}\|^{r}\\&
=\langle f(|T^{2}|)k_{\lambda}, k_{\lambda}\rangle^{\frac{r}{2}}\langle g(|(T^{2})^{*}|)k_{\lambda}, k_{\lambda}\rangle^{\frac{r}{2}}\\&
\leq\frac{1}{p}\langle f(|T^{2}|)k_{\lambda}, k_{\lambda}\rangle^{\frac{rp}{2}}+\frac{1}{q}\langle g(|(T^{2})^{*}|)k_{\lambda}, k_{\lambda}\rangle^{\frac{rq}{2}}\\&
\leq \frac{1}{p}\langle f(|T^{rp}|)k_{\lambda}, k_{\lambda}\rangle+\frac{1}{q}\langle g(|(T^{rq})^{*}|)k_{\lambda}, k_{\lambda}\rangle\\&
=\langle \frac{1}{p}f(|T^{rp}|)+\frac{1}{q}g(|(T^{rq})^{*}|)k_{\lambda},k_{\lambda}\rangle.
\end{align*}
From proof of Theorem , we have
\begin{align*}
|\langle Tk_{\lambda},k_{\lambda}\rangle|^{2r}\leq\frac{1}{2}\left(\|Tk_{\lambda}\|^r+\|T^{*}k_{\lambda}\|^{r}+\langle \frac{1}{p}f(|T^{rp}|)+\frac{1}{q}g(|(T^{rq})^{*}|)k_{\lambda},k_{\lambda}\rangle\right).
\end{align*}
By taking supremun over unit vector $k_{\lambda}$, we get desired inequality.
\end{proof}
\begin{theorem}
Let $T\in {\mathcal B}({\mathcal H})$. Then for any $\nu\in[0,1]$ and $t\in \mathbb{R}$,
\begin{align*}
\|T\|^{2}\leq((1-\nu)^{2}+\nu^{2})\textbf{ber}^{2}(T)+\nu\|T-tI\|^{2}+(1-\nu)\|T-itI\|^{2}.
\end{align*}
\end{theorem}
\begin{proof}
We use the following inequality, which obtained in 
\begin{align*}
&(\nu\|tb-a\|^{2}+(1-\nu)\|itb-a\|^{2})\|b\|^{2}\\&\geq\|a\|^{2}\|b\|^{2}-((1-\nu)Im\langle a,b\rangle+\nu Re\langle a,b\rangle)^{2} (\geq0)
\end{align*}
for any $a,b\in \mathcal H$, $\nu\in[0,1]$ and $t\in \mathbb{R}$.
So,
\begin{align*}
\|a\|^{2}\|b\|^{2}&\leq((1-\nu)Im\langle a,b\rangle+\nu Re\langle a,b\rangle)^{2}\\&+(\nu\|tb-a\|^{2}+(1-\nu)\|itb-a\|^{2})\|b\|^{2}\\&
\leq((1-\nu)^{2}+\nu^{2})|\langle a,b\rangle|^{2}+(\nu\|tb-a\|^{2}+(1-\nu)\|itb-a\|^{2})\|b\|^{2}.
\end{align*}
Choosing $a=Tk_{\lambda}, b=k_{\lambda}$  with $\|k_{\lambda}\|=1$, we get
\begin{align*}
\|Tk_{\lambda}\|^{2}\|k_{\lambda}\|^{2}&\leq((1-\nu)Im\langle Tk_{\lambda},k_{\lambda}\rangle+\nu Re\langle Tk_{\lambda},k_{\lambda}\rangle)^{2}\\&+(\nu\|tk_{\lambda}-Tk_{\lambda}\|^{2}+(1-\nu)\|itk_{\lambda}-Tk_{\lambda}\|^{2})\|k_{\lambda}\|^{2}\\&
\leq((1-\nu)^{2}+\nu^{2})|\langle Tk_{\lambda},k_{\lambda}\rangle|^{2}\\&+(\nu\|tk_{\lambda}-Tk_{\lambda}\|^{2}+(1-\nu)\|itk_{\lambda}-Tk_{\lambda}\|^{2})\|k_{\lambda}\|^{2}.
\end{align*}
Taking the supremum over all unit vectors $k_{\lambda}$, we get the desired result.
\end{proof}




\bibliographystyle{amsplain}

\end{document}